\documentclass[12pt,a4paper]{amsart}
\usepackage{amsthm}
\usepackage{amsmath}
\usepackage{amsfonts}
\usepackage{amssymb}
\usepackage[left=2.5cm, right=2.5cm, top=3.0cm, bottom=3.0cm]{geometry}
\usepackage{color}
\usepackage{bbm}
\usepackage[arrow, matrix, curve]{xy}
\usepackage{graphicx}
\usepackage[colorlinks=true,linkcolor=green]{hyperref}
\makeatletter
\renewcommand*{\eqref}[1]{
	\hyperref[{#1}]{\textup{\tagform@{\ref*{#1}}}}}
\makeatother

\newcommand{\inj}{\mathrm{inj}}

\newcommand{\cri}{\mathrm{cr}}

\newcommand{\beq}{\begin{equation}}
\newcommand{\beqn}{\begin{equation*}}

\newcommand{\eeq}{\end{equation}}
\newcommand{\eeqn}{\end{equation*}}

\newcommand{\R}{\mathbb{R}}

\newcommand{\inter}{\mathcal{I}}

\newcommand{\q}{\mathbb{Q}}

\newcommand{\z}{\mathbb{Z}}
\theoremstyle{plain}
\newtheorem{theorem}{Theorem}
\newtheorem{proposition}{Proposition}
\newtheorem{corollary}{Corollary}
\newtheorem{lemma}{Lemma}

\theoremstyle{definition}
\newtheorem{definition}{Definition}

\theoremstyle{remark}
\newtheorem{remark}{Remark}
\newtheorem{number-env}[theorem]{}
\title[Geodesic loops, orthogonal geodesic chords]{Geodesic 
loops and orthogonal geodesic chords
without self-intersections} 
\author{Hans-Bert Rademacher}
\address{Mathematisches Institut, 
Universit{\"a}t Leipzig, D--04081 Leipzig, Germany}
\email{rademacher@math.uni-leipzig.de}
\urladdr{www.math.uni-leipzig.de/\symbol{126}rademacher}
\date{2025-08-14}
\subjclass[2020]{53C22, 58E10}
\keywords{orthogonal geodesic chords without self-intersections,
	perturbation of Riemannian metrics,
	non-degenerate orthogonal geodesic chord, 
	orthogonal-tangent geodesic chord,
generic Riemannian metrics}
\begin{document}
\begin{abstract}
	We show that for a generic Riemannian metric on 
	a compact manifold of dimension $n\ge 3$
	all geodesic loops based at a fixed point
	have no self-intersections. We also 
show that for an open and dense subset of
the space of
Riemannian metrics on an
$n$-disc with $n \ge 3$ 
and with a striclty convex boundary
there are $n$ geometrically
distinct orthogonal geodesic chords 
without self-intersections.
We use a perturbation result for intersecting geodesic
segments of the author~\cite{R2022} and
a genericity statement due to Bettiol and
Giamb\`o~\cite{BG2010}
and existence results for orthogonal
geodesic chords by
Giamb\`o, Giannoni, and 
Piccione~\cite{GGP2018}.
\end{abstract}
\maketitle
\baselineskip 20pt
\section{Results}
\label{sec:results}

It is a well known result
that on any compact smooth manifold $M$ 
with tangent bundle $TM$ 
endowed with a Riemannian metric $g$ there exists a non-trivial
closed geodesic.
If the manifold is not simply-connected a shortest non-contractible curve
is a closed geodesic. In the simply-connected
case this result is due to
Lusternik and Fet~\cite{LF}, cf.~\cite[Sec.2.4]{Kl}.
We call a closed geodesic 
$c: S^1=\R/\z \longrightarrow M$ \emph{prime} if it is not the covering
of a shorter closed curve, then the velocity field
$c': S^1 \longrightarrow TM$ defines an injective map.
We call a closed geodesic \emph{simple} resp.
\emph{without self-intersections} if the map $c:S^1\longrightarrow M$
is injective. In particular a simple closed geodesic is prime.
The shortest noncontractible smooth curve on a manifold is a
simple closed geodesic. 
On one hand any closed surfaces carries a simple closed geodesic,
cf.~\cite[Rem.~5]{R2022}.
On the other hand it is an open question whether any Riemannian
manifold $M$ of dimension $n=\dim (M)\ge 3$ carries a simple
closed geodesic, cf. \cite[Introduction]{AHS1999}.
The situation is quite different for a generic metric
in dimensions $\ge 3.$
The author has shown recently that 
that for a generic Riemannian metric, resp. a generic
reversible Finsler metric,
on a compact manifold of dimension $n\ge 3$ all prime
closed
geodesics have no self-intersections, 
and they do not intersect each other, 
cf.~\cite[Thm.1]{R2022}. 

One  ingredient of the proof is the following
perturbation result for intersecting geodesic segments:
If $N\ge 2$ geometrically distinct geodesic
segments
$c_j:[-2\eta,2\eta]\longrightarrow M, j=1,\ldots,N$ 
parametrized by arc length intersect
in the common midpoint $p=c_j(0)$ for
sufficiently small $\eta>0$
then one can perturb away the intersection of the geodesic
segments. For a precise statement 
see Lemma~\ref{lem:perturbation}.

The other ingredient is the
\emph{bumpy metric theorem} 
for closed geodesics stating that
for a generic Riemannian metric
on a compact manifold all closed geodesics are non-degenerate
in the sense of 
Morse theory.
A closed geodesic $c: S^1\longrightarrow M$
is called \emph{non-degenerate} if there is no non-trivial
periodic Jacobi field along $c$ which is orthogonal to
the velocity field $c'.$
The statement is due to Abraham~\cite{Ab}, proofs were given
by Anosov~\cite{An}, and White~\cite[Thm~2.2]{Wh}, see also
the book by Moore~\cite[\S 2.7]{Mo}.
Hence the energy functional 
$E: \Lambda M\longrightarrow \R, E(\sigma)=\frac{1}{2}
\int_0^1 g(\sigma'(t),\sigma'(t))\,dt$
on the Hilbert manifold
$\Lambda M$ of closed curves 
$\sigma:S^1=[0,1]/\{0,1\}\longrightarrow M$
(called the \emph{free loop space})
is a Morse-Bott function whose critical set is 
the disjoint union of orbits of closed
geodesics under the canonical action of the orthogonal
group $\mathbb{O}(2)$ in dimension $2$ acting on
$S^1=\R/\mathbb{Z}=[0,1]/\{0,1\}.$
It also follows that
for any $L>0$ there are only finitely many geometrically distinct
closed geodesics with length $\le L.$
It is a consequence of the perturbation result and the
bumpy metric theorem that for a generic Riemannian metric
on a compact manifold of dimension $n\ge 3$ all prime
closed
geodesics have no self-intersections, i.e. are \emph{simple,}
and they do not intersect each other, 
as shown by the author, cf.~\cite[Thm.1]{R2022}. 
Note that on a surface with a non-degenerate closed geodesic
with self-intersection it is not possible to perturb away the
self-intersection.

Instead of considering closed resp. periodic geodesics one can
study different boundary conditions for geodesic segments. In this paper we
study first
for points $p,q\in M$ geodesic segments 
$c:\inter \longrightarrow M$ 
joining $p=c(0)$ and $q=c(1).$
Here we use as parameter space the unit interval $\inter=[0,1]$ if
$p\not=q$ and the circle $\inter=S^1=\R/\z=[0,1]/\{0,1\}$ in case of
a \emph{geodesic loop,} i.e. $p=q.$ 
Note that in the case of a geodesic loop $c:S^1\longrightarrow M$
only the map $c:[0,1]\longrightarrow M$ is smooth,
since in general for a geodesic loop in contrast to a 
closed resp. periodic geodesic we have $c'(0)\not=c'(1).$
We call the geodesic
segment $c:\inter\longrightarrow M$ 
\emph{simple,} resp.
\emph{without self-intersection,} if the map 
$c:\inter \longrightarrow M$ is injective.
We call two geodesic segments 
$c_1,c_2:\inter\longrightarrow M$
\emph{
	geometrically distinct,}
if 
$c_1(\inter)\not=c_2(\inter).$
We say that two geometrically distinct geodesic
segments $c_1,c_2:\inter \longrightarrow M$
do not have an \emph{interior intersection,}
if $c_1((0,1))\cap c_2((0,1))=\emptyset.$
It follows from a general
result by Bettiol and
Giamb\`o presented in ~\cite[Theorem 5.10]{BG2010} that for a generic Riemannian
metric on a compact manifold $M$ the geodesic segments, resp.
geodesic loops, 
are non-degenerate
in the sense of Morse theory, cf. Section~\ref{sec:boundary}. 
We show in this paper that we can
use the perturbation result Lemma~\ref{lem:perturbation} and
the genericity result of \cite[Theorem 5.10]{BG2010} 
to show the following result:
\begin{theorem}
	\label{thm:generic-loop}
	Let $M$ be a compact manifold of dimension $n\ge 3.$
	Let $p,q\in M$ be points in $M,$ where we also allow $p=q.$
	For a $C^k$-generic Riemannian metric on $M$ with $k\ge 2$
	all geodesic segments joining
	$p$ and $q$ are non-degenerate and simple, and 
	geometrically distinct geodesic segments do not
	have an interior intersection.
\end{theorem}
Note that Theorem~\ref{thm:generic-loop} also shows that for 
points $p,q$ for a generic metric there is no
closed geodesic passing through $p$ and $q,$ 
cf. Lemma~\ref{lem:closed-geodesic}.
It even shows that for a generic metric
any geodesic through $p$ and $q$ does not go through
$p$ resp. $q$ again.
In the particular case $p=q$ it follows from 
Theorem~\ref{thm:generic-loop}
that all
geodesic loops for a generic metric are
\emph{prime,}
i.e. there is no shorter geodesic loop, which
is a part of the given geodesic loop,
cf. \cite{Me}.
Hence for a generic metric for any two geodesics
$c,d:[0,1]\longrightarrow M, c(0)=d(0)=p,c(1)=d(1)=q$
we have 
either $c((0,1))=d((0,1))$ or 
$c([0,1])\cap d([0,1])=\{p,q\}.$

As a consequence of Theorem~\ref{thm:generic-loop}
we obtain the following
\begin{corollary}
	\label{cor:loops}
	Let $M$ be a simply-connected and compact manifold
	of dimension $n\ge 3$ 
	and $p,q \in M.$ 
	For a $C^k$-generic Riemannian metric on $M$ with $k\ge 2$ all
	geodesic segments joining $p$ and $q$ are 
	non-degenerate and simple,
	and geometrically distinct ones do not have interior
	intersection.
	For the number
	$N_g(p,q;t)$ of 
	geometrically distinct and simple
	geodesic segments joining $p$ and
	$q$ of length $\le t$ we obtain
	$\liminf_{t\to \infty}(N_g(p,q;t)/t)>0.$
\end{corollary}
This is a quantitative version of  
Serre's result~\cite[Ch.IV., Sec.7, Prop.14]{Se}
giving a
result concerning \emph{geometrically distinct} geodesic
segments between two fixed points.
Therefore Corollary~\ref{cor:loops}
is also satisfying from a geometric point of view,
cf. the discussion
in the doctoral thesis by Mentges~\cite[Sec.1.2]{Me} or by
Tanaka~\cite[Problem B]{Ta}.
On a surface, i.e. for $n=2,$ we cannot expect to remove intersections
of geodesic segments by a small perturbation. But for a 
point $p$ and a
generic metric on a closed surface,
all geodesic loops based at $p$ are prime, cf. Proposition~\ref{pro:surface}.

Another boundary condition for geodesic segments is
the following: Let $(M,g)$ be a compact Riemannian
manifold with smooth boundary $\partial M$ 
and interior $\textrm{Int}(M)=M-\partial M$ then a 
\emph{geodesic chord} $c:[0,1]\longrightarrow M$
is a geodesic segment with
$c(0),c(1)\in \partial M$ and $c(t)\in \textrm{Int}(M)$
for all $t\in (0,1).$ It is called 
an \emph{orthogonal geodesic chord} if
$c$ meets the boundary $\partial M$
orthogonally. 
Orthogonal geodesic chords are also called
\emph{free boundary geodesics,}
cf. for example~\cite{Do2023} or \cite{Zh}.
A chord $c:[0,1]\longrightarrow M$ is called an
\emph{orthogonal-tangent geodesic chord}
if it is tangent to the boundary at $c(0)$ and
orthogonal to the boundary at $c(1).$
As a consequence of the genericity result
by Bettiol and Giamb\`o~\cite{BG2010} we show:
\begin{theorem}
	\label{thm:generic-ogc}
	Let $M$ be a compact manifold of dimension $n\ge 3$
	with smooth boundary
	$\partial M.$ For a $C^k$-generic Riemannian metric 
	on $M$ with $k\ge 2$
	all 
	orthogonal geodesic chords are
	non-degenerate and have no self-intersections.
	In addition they all do not intersect each other. 
\end{theorem}
See Section~\ref{sec:boundary-conditions}
for the definition of non-degeneracy for
orthogonal geodesic chords.

Lusternik and Schnirelmann~\cite{LS} have shown that
a bounded domain in $\R^n, n\ge 2$ with a smooth and convex
boundary has at least $n$ \emph{double normals,}
i.e. lines intersecting the boundary orthogonally.
This was generalized for a Riemannian metric,
Bos~\cite[p.431]{Bo} showed that 
on an $n$-disc $D^n$ with a Riemannian metric
whose boundary $\partial D^n$ is strictly
convex there are
$n$ geometrically distinct orthogonal geodesic
chords. Using a generalization of 
this result due to
Giamb\`o, Giannoni, and 
Piccione~\cite[Theorem 1.1]
{GGP2018} we obtain from
Theorem~\ref{thm:generic-ogc}:
\begin{theorem}
	\label{thm:geod-chords}
	For $n\ge 3$ let $\mathcal{G}(D^n)$ be the space of 
	Riemannian metrics on the $n$-disc
	$D^n$ with the strong $C^k$-topology with $k\ge 2.$
	Let $\mathcal{G}_1=\mathcal{G}_1(D^n)$ 
	be the set of Riemannian metrics on 
	$D^n$ for which there is no
	orthogonal-tangent geodesic chord.
	Then $\mathcal{G}_1$
	is an open and non-empty
	subset
	of $\mathcal{G}(D^n).$
	
	Furthermore
	there is an open subset $\mathcal{G}_2\subset\mathcal{G}_1$
	whose closure $\overline{\mathcal{G}}_2$
	coincides with the closure $\overline{\mathcal{G}}_1,$
	i.e. $\overline{\mathcal{G}}_1=\overline{\mathcal{G}}_2,$
	with the following property:
	For all Riemannian metrics $g \in \mathcal{G}_2$
	there are $n$ geometrically distinct geodesic 
	orthogonal chords, which are non-degenerate and
	have no self-intersection, and do not intersect each other.
\end{theorem}
Here we note that geometrically distinct geodesic
chords cannot intersect at a boundary point.
The boundary $\partial D^n$ is called 
\emph{strictly convex,} if
all principal curvatures $\kappa_i, i=1,\ldots,n-1$ of the
boundary with respect to the inward pointing unit normal,
are positive, i.e. $\kappa_i > 0, i=1,\ldots,n-1.$
In particular there is no geodesic in the disc
touching the boundary
$\partial D^n$ and hence
there is no orthogonal-tangent 
geodesic chord. Therefore we obtain
from Theorem~\ref{thm:geod-chords}
the following result in the spirit of
Bos' theorem~\cite[p.431]{Bo}:
\begin{corollary}
	\label{cor:bos}
	There is an open and dense subset 
	$\mathcal{G}_*$ of the 
	set $\mathcal{G}_+(D^n)$ of Riemannian metrics on
	the
	$n$-dics $D^n$ of dimension
	$n\ge 3,$  for which the boundary $\partial D^n$ is 
	strictly convex, with the following property:
	For all Riemannian metrics
	$g\in \mathcal{G}_*,$ there are 
	at least $n$ geometrically distinct orthogonal geodesic
	chords without self-intersections. 
	
	Furthermore, these geodesic chords
	do not intersect each other.
\end{corollary} 
Using a curve shortening construction called 
\emph{disc flow on surfaces,}
Hass and Scott~\cite[Thm.4.2]{HS} have shown that a Riemannian
$2$-disc with strongly convex boundary carries at least
two geometrically distinct orthogonal geode\-sic chords
without self-intersections.
For the proof one can also use the 
\emph{curve shortening flow,} 
as shown by  Ko~\cite{Do2023}.
With a curve shortening flow, Li~\cite{Li} presents a proof of
the above
cited result by Lusternik and Schnirelmann about the
existence of $n$ double normals of a strictly convex 
compact domain in $\R^n.$

\smallskip

{\sc Acknowledgement.} 
I am grateful to the referee for 
the helpful comments and suggestions.
\section{Intersection of geodesic segments and perturbation}
\label{sec:intersection}
For the intersection points of geodesic segments we obtain the
following result which is similar to the corresponding result for
closed geodesics, cf.~\cite[Lem.2]{R2022} resp.
\cite[Sec.2.3]{BG2010}.
Recall from the Introduction that we denote a geodesic
segment $c:\inter\longrightarrow M$ joining $p=c(0),q=c(1)$
using the parameter space $\inter=[0,1]$ if $p\not=q$
and $\inter=S^1=[0,1]/\{0,1\}$ for $p=q.$
\begin{lemma}
	\label{lem:intersection}		
	\emph{(a)} Let $\gamma:\inter\longrightarrow M$
	be a geodesic segment, which is not the parametrization
	of a closed (periodic) geodesic of
	period $\le 1.$ Then the set of \emph{interior double points,}
	resp.
	\emph{interior self-intersection points} 
	\begin{equation*}
		DP_1(\gamma):=\{\gamma(t)\,;\,
		t\in (0,1), \#\gamma^{-1}(\gamma(t))\ge 2\}
	\end{equation*}
	of the geodesic segment $\gamma$ is finite.
	
	\smallskip
	
	\emph{(b)} Let $\gamma,\delta:
	[0,1]\longrightarrow M$ be geometrically distinct
	geodesic segments with $\gamma(0)=\delta(0)=p,
	\gamma(1)=\delta(1)=q.$ 
	Then the set of 
	\emph{interior intersection points}
	\begin{equation*}
		I_1(\gamma,\delta):=\gamma((0,1))
		\cap \delta((0,1))
	\end{equation*}
	is finite.	
\end{lemma}
\begin{remark}
	(a) Note that in case (a) for a geodesic segment
	$\gamma:\inter \longrightarrow M$ the initial point
	$c(0)$ is an interior self-intersection point, if
	there is a parameter $s\in (0,1)$ with $c(s)=c(0).$
	
	\smallskip
	
	(b)
	The geodesic segment $\gamma:\inter\longrightarrow M$
	has \emph{no interior self-intersection,} resp. is called \emph{simple,}
	if the map $\gamma$ is injective, i.e.
	if the set $DP_1(\gamma)$ of interior double points is empty. 
	The geodesic segments $\gamma,\delta$ joining $p,q$
	\emph{do not have interior intersection,}
	if $
	I_1(\gamma,\delta)=\emptyset.$
\end{remark}
\begin{proof}
	Let $\inj>0$ be the injectivity radius of $(M,g).$ We conclude from
	$\gamma(t)=\gamma(s)$ and $t<s$ that $ (s-t)L(\gamma)\ge \inj.$ This implies that
	for all $t\in [0,1]:
	\# \gamma^{-1}(\gamma(t))\le L(\gamma)/\inj+1.$
	
	\smallskip
	
	(a) If the set $DP_1(\gamma)$ is not finite, then there exist sequences
	$s_j<t_j\in [0,1]$ converging to $s^*,t^*\in [0,1]$ with
	$s_j\not=s^*,t_j\not=t^*, (t_j-s_j)L(\gamma)\ge \inj$ for all
	$j,$ 
	$(t^*-s^*)L(\gamma)\ge \inj,$
	and 
	$\gamma(s_j)=\gamma(t_j)$ for all $j\ge 1.$ 
	For $p_1=\gamma(t^*)=\gamma(s^*)$ set $v=c'(t^*), w=c'(s^*).$
	The geodesic segment $\gamma$ is parametrized proportional
	to arc length, hence $\|v\|=\|w\|.$
	Since $c(t_j)=\exp_p((t_j-t^*)v)=c(s_j)=\exp_p((s_j-s^*)w)$ 
	we obtain for sufficiently large $j$ with 
	$|t_j-t^*|\|v\|,|s_j-s^*|\|w\|<\inj$ that 
	$(t_j-t^*)v=(s_j-s^*)w$ holds, i.e. $v=\pm w.$ 
	If $v=w$ then $\gamma(t)=\gamma(t+(t^*-s^*))$ for all $t,$ i.e.
	$\gamma $ is a closed geodesic, which is not prime.
	If $v=-w$ we conclude
	$\gamma((t^*-s^*)/2-t)=\gamma((t^*-s^*)/2+t)$ for all $t, $
	showing that $\gamma$ is not smooth at $(t^*-s^*)/2.$

	\smallskip
	(b) The argument is similar to (a): 
	If the set $I_1(\gamma,\delta)$ is infinite then there are sequences
	$s_j,t_j\in [0,1]$ converging to $s^*,t^*$
	with $s_j\not=s^*, t_j\not=t^*$ and 
	$\gamma(s_j)=\delta(t_j)$ for all $j\ge 1.$ 
	Let $p=\gamma(s^*)=\delta(t^*)$ and $v=\gamma'(s^*),
	w=\delta'(t^*).$
	Since $\gamma$ and $\delta$ are geometrically distinct, we have
	$v/\|v\|\not=\pm w/\|w\|.$
	Since $\gamma(t_j)=\exp_p((t_j-t^*)v)=
	\delta(s_j)=\exp_p((s_j-s^*)w)$ we conclude for sufficiently large
	$j$ with
	$|t_j-t^*|\|v\|,|s_j-s^*|\|w\|<\inj$ that 
	$(t_j-t^*)v=(s_j-s^*)w$ holds, i.e. $v/\|v\|=\pm w/\|w\|,$ which is 
	a contradiction.
\end{proof}
We obtain from the perturbation result~\cite[Lem.3]{R2022}:
\begin{lemma}
	\label{lem:perturbation}
	Let $(M,g)$ be a compact Riemannian manifold 
	of dimension $n \ge 3$ and 
	$\eta>0$ be a positive number, such that the injectivity radius
	$\inj$ satisfies $\inj \ge 4\eta.$
	For a point $p \in M$ let 
	$c_j:[-2\eta,2\eta]\longrightarrow M, j
	=1,\ldots,k$
	be $k\ge 2$ geometrically
	distinct geodesic segments
	parametrized by arc length with 
	common midpoint $p=c_j(0),
	j=1,\ldots, k.$
	
	In any neighborhood $\mathcal{U} \ni g$
	with respect to the strong $C^k$-topology with $k\ge 2,$
	there is a Riemannian metric $\overline{g}\in \mathcal{U}$ 
	coinciding with $g$
	outside the geodesic ball $B_p(\eta)$ of radius $\eta$
	centered at $p,$ 
	carrying geodesic
	segments
	$\overline{c}_j:[-2\eta,2\eta]\longrightarrow M,
	j=1,\ldots,k$ parametrized
	by arc length, which coincide with $c_j$ outside $B_p(\eta).$
	
	These geodesic segments are simple and do not intersect each
	other.
\end{lemma}
A similar proof as in
Lemma~\cite[Lem.3]{R2022} shows the following
\begin{lemma}
	\label{lem:perturbation2}
	Let $c:[0,1]\longrightarrow M$
	be a geodesic loop with $p=c(0)=c(1)$
	of the Riemannian metric $g$ on the differentiable
	manifold $M$
	with $v=c'(0)\in T_pM, w=c'(1)\in T_pM.$
	Then there exists a continuous
	one parameter perturbation 
	$g_s, s\in [0,s_1]\in \mathcal{G}(M)$ 
	in the space of Riemannian metrics with
	the strong $C^k$-topology with $k\ge 2$ 
	of the metric $g=g_0,$ 
	and smooth functions
	$s\in [0,s_1]\mapsto v(s), w(s) \in T_pM,$
	such that $c_s:[0,1]\longrightarrow M$ is a geodesic
	loop with $p=c(0)$ of
	the Riemannian metric $g_s$ 
	with $v(s)=c_s'(0), w(s)=c_s'(1),$ and $v(s),v$ (resp.
	$w(s),w$) are linearly independent for $s>0.$ 
	The metric $g_s$ agrees with
	$g$ outside an arbitrarily small geodesic
	ball around $p.$
\end{lemma}	
\section{Geodesics with product boundary conditions}
\label{sec:boundary}
\label{sec:boundary-conditions}
The following concept is introduced in \cite[Sec.4]{BG2010},
cf. also~\cite{Gr}.
Let $M$ be a differentiable manifold.
Let $\mathcal{P}\subset M \times M$ be a 
compact submanifold. This is called a 
\emph{general boundary condition} on $M$ if 
for some Riemannian metric $g$ on $M$ the 
pseudo-Riemannian metric
$\overline{g}=
g \oplus (-g)$ on $M \times M$ induces a non-degenerate
pseudo-Riemannian metric on $\mathcal{P},$
cf. ~\cite[Def.4.2]{BG2010}.
It is shown in \cite[Prop.4.1]{BG2010}
that the subset $\mathcal{A}_{\mathcal{P}}
\subset \mathcal{G}(M)$ of the set
of Riemannian metrics on $M$ for which the 
restriction of
the pseudo Riemannian metric $\overline{g}$ onto the submanifold
$\mathcal{P}$ is non-degenerate, is an open subset.
The \emph{second fundamental form} of $\mathcal{P}$
in the normal direction $\eta \in T\mathcal{P}^{\perp}$ 
is given by
\begin{equation*}
	S^{\mathcal{P}}_{\eta}(v,w)=\overline{g}
	\left(\nabla_v^{\overline{g}} \overline{w},\eta
	\right),
\end{equation*}
cf.~\cite[p.346]{BG2010}. Here $\overline{\nabla}$ is the
Levi-Civita connection on $M\times M$ with metric
$\overline{g}.$ The vector field $\overline{w}$ is a smooth
extension of $w$ to a smooth vector field along $\mathcal{P}.$
Since the restriction of $\overline{g}$ to $\mathcal{P}$ is
non-degenerate we can identify the symmetric bilinear form
with the self-adjoint \emph{shape operator} 
$(S_{\eta}^{\mathcal{P}})_p:
T_p\mathcal{P}\longrightarrow T_p\mathcal{P}$
with $\overline{g}((S_{\eta}^{\mathcal{P}})_pv,w)=
S_{\eta}^{\mathcal{P}}(v,w)$ for all $v,w \in T_p\mathcal{P}.$
Let $H^1(I,M)$ be the set of all curves of Sobolev
class $H^1$ in $M,$ i.e. of absolutely continuous curves
with finite energy. Then the subspace
\begin{equation*}
	\Omega_{\mathcal{P}}(M)=
	\left\{
	\gamma \in H^1(I;M)\,;\, (\gamma(0),\gamma(1))\in \mathcal{P}
	\right\}
\end{equation*}
is a Hilbert manifold. The tangent space
$T_{\gamma}\Omega_{\mathcal{P}}(M)$
can be identified with the completion of the 
space of smooth vector fields
$v$ along $\gamma$ with $(\gamma(0),\gamma(1))\in
T_{(\gamma(0),\gamma(1))}\mathcal{P}$
with respect to the induced metric $g_1=\langle .,.
\rangle_1.$ It is defined on
$T_{\gamma}\Omega_{\mathcal{P}}(M)$
by
$$
\langle v,w\rangle_1=
\int_0^1 g\left( v(t),w(t)\right) \,dt +
\int_0^1 g\left(\nabla v(t),\nabla w(t)
\right)\, dt\,.
$$
A geodesic 
$\gamma:I\longrightarrow M$ 
is called a $(g,\mathcal{P})-$\emph{geodesic}
if $(\gamma'(0),\gamma'(1))\in T_{(\gamma(0),\gamma(1))}
\mathcal{P}^{\perp}.$ Here the orthogonality is with
respect to the metric $\overline{g}.$
It is shown in \cite[Sec. 5]{BG2010}
that a curve $\gamma$ is a critical point of
the energy functional $E_{g}:\Omega_{\mathcal{P}}(M)
\longrightarrow \R,$
\begin{equation}
	\label{eq:energy}
	E_{g}(\gamma)=\frac{1}{2}\int_0^1 g(\gamma'(t),
	\gamma'(t))\,dt
\end{equation}
if and only if it is a $(g,\mathcal{P})$-geodesic.
The kernel of the index form
is the space of Jacobi fields
$Y \in T_{\gamma}\Omega_{\mathcal{P}}(M)$ such that
\begin{equation}
	\left(
	\nabla  Y(0),\nabla Y(1)
	\right)	+
	S_{(\gamma'(0),\gamma'(1))}(Y(0),Y(1))
	\in T_{(\gamma(0),\gamma(1))}\mathcal{P}^{\perp}.	
\end{equation}
These Jacobi fields are also called $\mathcal{P}$-Jacobi fields.
This setting is also considered by Grove
in the context of
\emph{isometry invariant geodesics,} it is shown in
~\cite[Thm.2.4]{Gr} that the 
energy functional~\eqref{eq:energy}
satisfies the Palais-Smale condition,
resp. condition (C), provided that 
the Riemannian manifold $(M,g)$ is complete and
the space $\mathcal{P}$ is compact.

A general boundary condition is called
\emph{admissible} if for every $g_0 \in \mathcal{A}_{\mathcal{P}}$
there exists an open neighborhood $\mathcal{V}$
of $g_0$ in $\mathcal{A}_{\mathcal{P}}$
and a positive real number $a>0$ such that
for all $g \in \mathcal{V}$ and all non-trivial
$(g,\mathcal{P})$-geodesics $\gamma$
we have: $L_g(\gamma)\ge a,$
cf. \cite[Def.4.5]{BG2010}.
Here a non-trivial $(g,\mathcal{P})$ geodesic is 
a non-constant curve.

We restrict ourselves to the following case:
\begin{definition}
	\label{def:product}
	Let $(M,g)$ be a complete Riemannian manifold
	of dimension $n$ 
	and $N_0,N_1\subset M$ connected 
	and compact submanifolds (without boundary)
	of dimension $\dim N_i=n_i\in \{0,1,\ldots,n-1\}, i=0,1.$
	Then the \emph{product boundary condition}
	is given as general boundary condition with
	$\mathcal{P}=N_0\times N_1.$
\end{definition}
In particular a $(g,N_0\times N_1)$-geodesic
is a geodesic $c:[0,1]\longrightarrow M$
with $c(l)\in N_l,
c'(l)\in T_{c(l)}^{\perp}N_l,l=0,1.$
The product metric
$g \oplus (-g)$ on $N_0\times N_1$ is non-degenerate,
hence a product boundary condition is always
non-degenerate, as defined in \cite[Sec.4]{BG2010}.
In case $\mathcal{P}=N\times N$
the $(g,N\times N)$-Jacobi fields are Jacobi fields
$Y \in T_{\gamma}\Omega_{N\times N}$ with
\begin{equation}
	\label{eq:Jacobi2}
	\nabla Y(l)+S_{\gamma(l)}(Y(l)) 
	\in T_{\gamma(l)}N^{\perp}\,, l=0,1\,,
\end{equation}
they form the kernel of the index form.
The $(g,N\times N)$-geodesic $\gamma$ is 
non-degenerate if the kernel is trivial,
i.e. if there is no non-trivial
$(g,N\times N)$-Jacobi field.
Here $S_x$ denotes the shape operator of the
submanifold $N\subset M$ with respect to a unit normal
field $\nu,$ i.e.
$$
S_x (v)=\nabla_{v}\nu\,.
$$
And $\nabla$ is the Levi-Civita connection of $(M,g).$
If $N_0=\{p\},N_1=\{q\},$ then a 
$(g,\{p\}\times \{q\})$-Jacobi field along a
$(g,\{p\}\times \{q\})$-geodesic segment $\gamma:[0,1]\longrightarrow M$
is a Jacobi field $Y$ vanishing at the end points,
i.e. $Y(0)=0, Y(1)=0.$ Hence $q$ is not a conjugate point
for $p$ along $\gamma$
if and only if
there is no $(g,\{p\}\times\{q\})$-Jacobi field.
\begin{lemma}
	\label{lem:admissible}	
	The following sets $\mathcal{P}$ are admissible product
	boundary conditions:
	
	\smallskip
	
	(a) Let $(M,g)$ be a compact Riemannian manifold
	(without boundary)
	of dimension $n\ge 2.$
	If $p,q\in M$ are two points, where we also allow
	$p=q,$ then $\mathcal{P}=\{p\}\times\{q\}$ is an
	admissible product boundary condition.
	
	\smallskip
	
	(b) If $(M,g)$ is a compact Riemannian manifold 
	of dimension
	$n\ge 2$ with smooth boundary $\partial M,$
	then $\mathcal{P}=\partial M \times \partial M$ is 
	an admissible product boundary condition.
	
	\smallskip
	
	(c) If $(M,g)$ is a compact Riemannian manifold
	(without boundary) of dimension $n\ge 2,$
	and 
	if $N \subset M$ is a compact submanifold
	of codimension one,
	then $\mathcal{P}=N \times N$ is an admissible
	product boundary condition.
\end{lemma}
\begin{remark}
	\label{rem:ogc-general}
	In case (b) we can consider the compact manifold
	$M$ with smooth boundary $\partial M$ 
	endowed with a Riemannian metric $g$ as subset
	of a complete Riemannian manifold $(M_1,g_1).$
	
\end{remark}

\begin{proof}
	(a) 
	We can choose $a=d(p,q)/2$ if $p\not=q,$
	and $a=\inj/2$ for $p=q,$ here
	$\inj $ is the injectivity radius.
	Here we use that the distance as well as the
	injectivity radius are continuous functions
	on $\mathcal{G}^r(M)$ for $r\ge 2.$
	
	\smallskip
	
	(b) and (c). Here the normal exponential
	map of the submanifold $\partial M$ resp.
	$N$ is defined 
	$\exp^{\perp}: \nu^+(\partial M)\longrightarrow M,$
	resp. $\exp^{\perp}: \nu(N)\longrightarrow M.$
	The set $\nu^+(\partial M)$ is the subset of the normal
	bundle of $\partial M$ of normal vectors pointing inward 
	and $\nu(N)$ is the normal bundle of $N\subset M.$
	Then there is a positive 
	and continuous function $g \in \mathcal{G}^r(M)
	\mapsto 
	\inj^{\perp}(g),$ 
	which is also called the injectivity radius of 
	the normal exponential map 
	$\exp^{\perp},$
	such that the 
	restriction onto the normal vectors
	of length $< \inj^{\perp}$ is a diffeomorphism
	onto its image.
	Any non-trivial
	$(g,\partial M\times \partial M)$-geodesic, resp.
	$(g,N\times N)$-geodesic $\gamma:[0,1]\longrightarrow
	M,$ satisfies $L_g(\gamma)\ge 2\inj^{\perp}(g).$
\end{proof}
We will discuss the following examples in detail:
If for some points $p,q\in M$ we have $N_0=\{p\}, N_1=\{q\},$
then a $(g,\{p\}\times\{q\})$-geodesic is a geodesic
segment joining $p$ and $q.$ Here we also allow
the case $p=q,$ i.e. the case of geodesic loops based
at $p.$ Another important case is
$N_0=N_1=\partial N$ is a compact smooth
boundary of an open subset $N$, then orthogonal
geodesic chords are 
$(g,\partial{N}\times \partial{N})$-geodesics.
And for a compact hypersurface
$N\subset M$ i.e. a compact hypersurface $N$ 
we consider $(g,N\times N)$-geodesics, i.e.
geodesics $c:[0,1]\longrightarrow M$
with $c(l)\in N, c'(l)\in T_{c(l)}^{\perp}N, l=0,1.$
\begin{lemma}
	\label{lem:generic-ogc}
	Let $M$ be a compact manifold of dimension
	$n\ge 2$ with smooth boundary $\partial M.$
	Then there exists a compact manifold $M_1$ without
	boundary, such that $M$ is an open subset of
	$M_1$ and $\partial M$ is a smooth submanifold of
	$M_1.$ 
\end{lemma}
\begin{proof}
	It is a standard construction that one can 
	\emph{add a cylindrical neck}
	to the manifold $M$ with smooth boundary,
	i.e. there exists a Riemannian manifold $M_0$
	with metric $g_0$
	such that $M$ is an open submanifold 
	and $g$ is the restriction of $g_0$ to $M.$
	Furthermore, for some positive $\epsilon >0$
	the normal exponential map
	of the hypersurface $\partial M$ defines 
	an injective immersion
	$\phi:(t,x) \in [-\epsilon,0]
	\times \partial M \longmapsto M_0.$
	The Riemannian metric on $M$ near $\partial M$
	then has the form
	$dt^2 + h_t$ for a smooth family 
	$t \in [-\epsilon, 0] \mapsto h_t \in
	\mathcal{G}(\partial M)$
	of Riemannian metrics $h_t$
	on $\partial M.$ 
	Then we define the manifold $M_0$ as the 
	disjoint union of $M$ and the cylinder
	$[-\epsilon,1]\times \partial M \ni (t,x)$
	where we identify $(t,x)$ and $\phi(t,x)$ for all
	$(t,x)\in (-\epsilon,0]\times \partial M.$
	We extend the Riemannian metric on 
	$[-\epsilon,1]\times \partial M$ as $dt^2+h_t$
	such that 
	$h_t=h_1$ for all $t\in[2/3,1].$
	Hence the set $\phi([2/3,1]\times \partial M)$
	is a \emph{cylindrical end,} which is isometric
	to the Riemannian product $[2/3,1]\times(\partial M,h_1).$
	This construction can be found for example
	in~\cite[Section 5]{GZ}.
	
	We take two copies of $(M_0,g_0)$ 
	and identify them 
	along the boundary $\partial M_0.$
	This way we obtain a compact Riemannian manifold
	$(M_1,g_1)$ which consists of two copies of $(M_0,g_0)$
	without boundary, such that
	$\partial M_0$ is a totally geodesic submanifold,
	which is the fixed point set of an isometric involution
	interchanging the two copies of $(M_0,g_0).$
\end{proof}
\begin{remark}
	\label{rem:openanddense}
	In the setting of Lemma~\ref{lem:generic-ogc}
	let $\phi: \mathcal{G}(M_1)\longrightarrow 
	\mathcal{G}(M)$ be the map defined by restricting
	the Riemannian metric $g$ on $M_1$ to the submanifold
	$M,$ i.e. $\phi(g)=g|M.$
	The following statements are an immediate consequence of
	the definition of the strong topology, in particular of
	the definition of a \emph{strong basic neighborhood}
	cf. \cite[Section 2.1]{Hi} applied to  the 
	space of Riemannian metrics on $M$ resp. $M_1.$
	After fixing a Riemannian metric 
	the space $\mathcal{G}(M)$ of Riemannian metrics
	of class $C^k,k\ge 2$ on the manifold $M$ endowed with
	with the strong topology can be seen as an 
	open subset in the Banach
	space of $C^k$-sections of the bundle of 
	symmetric bilinear forms on the manifold $M.$
	If $\mathcal{U}$ is an open
	and dense subset of $\mathcal{G}(M_1)$ then $\phi(\mathcal{U})$
	is an open and dense subset of $\mathcal{G}(M).$
	And if $\mathcal{V}$ is a residual subset of
	$\mathcal{G}(M_1),$ then the restriction 
	$\phi(\mathcal{V})$ is a residual subset of
	$\mathcal{G}(M).$	
\end{remark}
\begin{lemma}
	\label{lem:closed-geodesic} Let $L>0.$
	
	(a) Let $(M,g)$ be a compact Riemannian manifold
	of dimension $n\ge 3$
	and $p,q\in M.$ 
	Then there is an open and dense set $\mathcal{G}^*[L]$ of
	Riemannian metrics such that for any $g\in \mathcal{G}^*[L]$
	no geodesic segment of length $\le L$ joining $p$ and $q$
	is part of 	a closed geodesic.
	
	\smallskip
	
	(b) Let $(M,g)$  a compact Riemannian manifold of 
	dimension $n\ge 2$ and let 
	$N\subset M$ be a compact submanifold of codimension $1.$
	Then there is an open and dense set $\mathcal{G}^*[L]$ of
	Riemannian metrics on $M$ such that for any $g\in \mathcal{G}^*[L]$
	no $(g,N \times N)$-geodesic of length $\le L$
	is part of 	a closed geodesic.
	
\end{lemma}
\begin{proof}

	Assume that $\gamma: S^1=[0,1]/\{0,1\}\longrightarrow M$ is a closed
	geodesic. We show that the set
	$A(\gamma):=
	\{t \in S^1; \gamma'(t)\in T_{c(t)}^{\perp}N_l\}, l=0,1$ is finite.
	Since the boundary condition is admissible 
	there is a positive $a>0$ such
	that for $s,t \in A, s\not=t$ we have
	$|t-s|L(\gamma)\ge a.$ This implies that $A$ is a
	finite set.
	
	For $L>0$ let $\mathcal{G}^*[L]$
	be the set of Riemannian metrics for which
	all closed geodesics with energy $\le L$ are non-degenerate
	and simple, cf.~\cite[Thm.1]{R2022}
	such that no restriction onto a subinterval defines
	- up to linear reparametrization - a
	geodesic segment between $p$ and $q$ in case (a)
	resp. a $(g,N\times N)$-geodesic in case (b).
	This set is open, we have to show that it is 
	a dense subset of $\mathcal{G}(M).$
	A metric $g \in \mathcal{G}^*[L]$ has only finitely
	many geometrically distinct prime closed geodesics 
	$c_j: S^1 \longrightarrow M, j=1\ldots,k$
	of energy
	$\le L$ and there is an open subset
	$\mathcal{U}\subset \mathcal{G}^*[L]$ and
	continuous functions $g\in \mathcal{U}\mapsto 
	v_j(g)\in TM$ such that
	the curve $t\in \R \mapsto 
	c_j(t)=\Phi_g^t(v_j(g)), t\in S^1$
	is a closed geodesic of $g$ and there are
	- up to geometric equivalence - no further
	closed geodesics of length $\le L$ of the metric $g,$
	cf.~\cite[p.12-13]{An}, resp. the Proof of 
	Part (b) of Theorem~\ref{thm:four}.
	Here $\Phi^t_g: TM \longrightarrow TM, t\in \R$ denotes
	the \emph{geodesic flow,} i.e. for $v\in T_pM,
	t \in \R \mapsto \Phi_g^t(v)\in TM$
	is the periodic orbit in the tangent bundle whose projection
	on the manifold is the geodesic $c:\R \longrightarrow
	M$ with initial point $p=c(0)$ and  initial
	direction $v=c'(0).$
	
	Hence the set $\widetilde{A}=\bigcup_{i=1}^k
	\{c_j(t)\,;\, t\in A(c_j)\}$
	is finite. 
	
	\smallskip
	
	(a) We can perturb the metric in a sufficiently small
	neighborhood of the points $p$ and $q$ using 
	Lemma~\ref{lem:perturbation}
	such that the perturbed closed geodesics of length $\le L$
	do not meet $p$ and $q.$
	
	\smallskip
	
	(b) We perturb the metric in a
	sufficiently small neighborhood of $p
	=c_j(t)\in \widetilde{A}, t\in A(c_j)$
	such that $c_j$ remains a closed geodesic but
	$c_j'(t)\not\in T_{c_j(t)}^{\perp}N_l$
	with respect to the perturbed metric.
	
	This is possible as the following Lemma shows.
\end{proof}
\begin{lemma}
	\label{lem:perturb-closed}
	Let $D^{n-1}_{2\eta}=\{x\in \R^{n-1}\,;\,
	\|x\|\le 2\eta\}$ be the Euclidean $(n-1)$-disc of radius
	$2\eta$ for $\eta>0.$
	On $(t,x)=(t,x_1,\ldots,x_{n-1})\in[-2\eta, 2\eta]\times D^{n-1}_{2\eta}$
	a Riemannian metric is given by
	$$dt^2+\sum_{i,j=1}^{n-1} g_{ij}(t,x)dx_i dx_j,$$
	hence $\gamma(t)=(t,0,\ldots,0), t \in [-2\eta,2\eta]$
	is a geodesic of $g$ parametrized by arc length
	orthogonal at $p=\gamma(0)=(0,\ldots,0)$
	to the hypersurface $\Sigma:=\{0\}\times D^{n-1}_{2\eta}.$
	Then in any neighborhood of the metric $g$ there is 
	a metric $\overline{g}$ for which $\gamma$ is still
	a geodesic parametrized by arc length
	and which is not
	orthogonal to $\Sigma$ at $p$
	with respect to $\overline{g}.$	
\end{lemma}
\begin{proof}
	Choose a smooth bump
	function $f:[-2\eta,2\eta]\longrightarrow [0,1]$
	with $f(t)=0$ for $|t|\ge \eta,$
	$f(t)=1 $ for $|t|\le \eta/2.$
	Let $s=t+\epsilon f(t)x_1.$
	For sufficiently small $\epsilon>0$ 
	with inverse function $t_{\epsilon}(s,x_1)$
	in the coordinates $s,x_1,\ldots,x_{n-1}$ the metric is
	of the form
	\begin{equation*}
		\frac{
			ds^2
			+2\epsilon f(t) ds dx_1+\epsilon^2 f^2(t) dx_1^2
		}{(1-\epsilon f'(t))^2}
		+\sum_{i,j=1}^{n-1} 
		g_{ij}(t_{\epsilon}(s,x_1),x_1,\ldots,x_{n-1})dx_i dx_j\,.
	\end{equation*}
	Then for sufficiently small $\epsilon>0$  
	the metric $\overline{g}$ given by
	$$
	ds^2+\sum_{i,j=1}^{n-1} 
	g_{ij}(t_{\epsilon}(s,x_1),x_1,\ldots,x_{n-1}) dx_i dx_j
	$$
	is sufficiently close go $g$ in the $C^k$-topology for 
	$k\ge 2$ and the
	the curve $x_1=\ldots=x_{n-1}=0$ coincides with $\gamma$
	and is a geodesic of $\overline{g}$ parametrized by
	arc length. At the point $p=\gamma(0)$ the curve
	is orthogonal to the hypersurface $s=0$ but not to the
	hypersurface $\Sigma=\{t=0\}.$	
\end{proof}
With the help of Lemma~\ref{lem:closed-geodesic} we can
prove the following main result, from which we can derive
the theorems stated in the Introduction:
\begin{theorem}
	\label{thm:four}
	Let $(M,g)$ be a compact Riemannian manifold
	of dimension $n\ge 3.$
	
	\smallskip
	
	(a)
	Let $M$ be a compact differentiable manifold of
	dimension $n\ge 3.$
	Let $p,q\in M.$ Then for a $C^k$-generic Riemannian metric
	with $k\ge 2$
	on $M$ all geodesic segments between $p$ and $q$
	are non-degenerate and simple, and any geometrically
	distinct geodesic segments between $p$ and $q$
	do not have interior intersection.
	
	\smallskip
	
	(b) Let $M$ be a compact differentiable manifold of
	dimension $n\ge 3$ with a smooth boundary $\partial M.$
	Then for a $C^k$-generic Riemannian metric on $M$
	with $k\ge 2$
	all orthogonal geodesic chords are non-degenerate and simple,
	and any two geometrically distinct orthogonal geodesic
	chords do not intersect.
	
	\smallskip
	
	(c) Let $M$ be a compact differentiable manifold 
	(without boundary) of
	dimension $n \ge 3$ and let $N\subset M$ be a compact
	submanifold of codimension $1.$ Then for a 
	$C^k$-generic metric on $M$ with $k\ge 2$ all
	$(g,N\times N)$-geodesics
	are non-degenerate and simple, and any two geometrically
	distinct 
	$(g,N\times N)$-geodesics do not intersect.
\end{theorem}
\begin{proof}
	We first consider the three cases 
	(a), (b), (c) at the same time,
	i.e. we set $\mathcal{P}=N_0\times N_1=\{p\}\times\{q\},$
	or $\mathcal{P}=\partial{M}\times \partial{M},$
	or $\mathcal{P}=N \times N.$
	
	It follows from Lemma~\ref{lem:closed-geodesic}
	and from the result~\cite[Thm.5.1]{BG2010},
	cf. also \cite[p.347]{BG2010},
	that for $L>0$ the set 
	$\mathcal{G}_{\mathcal{P}}(M)[L]$ of metrics 
	$		g\in \mathcal{G}_{\mathcal{P}}(M)$
	for which all $(g,\mathcal{P})$-geodesics
	of length $\le L$ are nondegenerate and are not part of
	a closed geodesic is an open
	and dense subset of $\mathcal{G}_{\mathcal{P}}
	=\mathcal{G}_{\mathcal{P}}(M).$
	This holds since in all three cases
	(a),(b),(c) the boundary
	conditions $\mathcal{P}$ are non-degenerate and
	admissible, cf. Lemma~\ref{lem:admissible}.
	
	Here one has to note that in case (b) one has to
	consider the differentiable manifold $M$ to be a
	subset of a compact manifold $M_1$ without boundary,
	cf. Lemma~\ref{lem:generic-ogc}.	
	Then one can apply~\cite[Thm.5.1]{BG2010}
	and conclude that the set 
	$\mathcal{G}_{\partial M\times \partial M}(M_1)[L]$ 
	of metrics $g\in \mathcal{G}(M_1)$ for which
	all $(g,\partial M\times \partial M)$-geodesics
	of length $\le L$ are non-degenerate is an
	open and dense subset of 
	$\mathcal{G}_{\partial M\times\partial M}(M_1).$
	Then the set
	$\mathcal{G}_{\partial M\times \partial M}(M)$
	contains all restrictions $g|M$ of metrics
	$g \in \mathcal{G}_{\partial M\times \partial M}(M_2)[L],$
	and is therefore also dense in $\mathcal{G}(M),$
	cf. Remark~\ref{rem:openanddense}.
	
	For any $L>0$ we define the set 
	$\mathcal{G}_{\mathcal{P}}^{*}(M)[L]$ of metrics 
	$g\in \mathcal{G}^L(M)$
	for which all 
	$(g,\mathcal{P})$-geodesics 
	of length $\le L$ are nondegenerate
	and simple and such that
	for any pair
	$\gamma,\delta$  of geometrically distinct
	$(g,\mathcal{P})$-geodesics of length $\le L$
	in case of $\mathcal{P}=\{p\}\times\{q\}$
	the interior intersection set
	$I_1(\gamma,\delta)$ is empty, and in the other
	two cases the intersection 
	$\gamma(I)\cap \delta(I)=\emptyset.$
	
	We conclude from~\cite[Thm.2.4]{Gr}
	that the energy functional
	$E_g:c \in \Omega_{\mathcal{P}}(M)
	\longmapsto E_g(c)\in \R$ satisfies the 
	Palais-Smale condition. We can apply this result since
	$M$ is compact, hence any Riemannian metric
	on $M$ is complete.
	
	Hence for $g\in \mathcal{G}_{\mathcal{P}}(M)[L]$
	the restriction
	$E_g:\Omega_{\mathcal{P}}(M)\cap \{E\le L\}
	\longrightarrow \R$
	is a Morse function with only finitely many critical
	points, i.e. there are only finitely many
	geometrically distinct 
	$(g,\mathcal{P})$-geodesics $c_j,j=1,\ldots,k$
	of energy $\le L.$ 
	
	The self-adjoint operator
	$H_c$ representing the hessian $d^2 E_g(c)$ at a 
	$(g,\mathcal{P}
	)$-geodesic $c$
	is of the form $H_c=\mathbbm{1}+k_c,$ i.e.  the sum of the identity and 
	a compact operator $k_c.$ 
	A proof can be found in\cite[Lem.2.5.2]{Kl}
	or in \cite[Sec.5.1]{BG2010}, where one can set
	$g_0=g_R,$ hence $A=\mathbbm{1}.$	
	Since $0$ is not an eigenvalue the
	inverse $H_c^{-1}$ is a bounded linear operator. 
	Therefore we can use the implicit function theorem
	for Banach manifolds, cf.~\cite[Thm.15.1]{De}
	and conclude:
	There is a neigbhourhood 
	$\mathcal{U}$ of $g$ and continuous functions
	$g\in \mathcal{U}\longmapsto 
	c_j[g] \in \Omega_{\mathcal{P}}(M),
	j=1,\ldots,k$ such that
	$c_j[g]$ is a $(g,\mathcal{P})$-geodesic of $g \in \mathcal{U}.$
	
	One can choose the open neighborhhod $\mathcal{U}$
	sufficiently small such that there are no 
	further
	$(g,\mathcal{P})$-geodesics of $g$ with $E_g(\gamma)\le L,$
	cf. the argument in~\cite[p.12-13]{An}.
	
	Since the geodesics $c_j$ are not part
	of a closed geodesic we conclude
	from Lemma~\ref{lem:intersection}:
	The union $DIP_1(L)$
	of the double points $DP_1(c_j), j=1,\ldots,k$ and the
	intersection points 
	$I(c_j,c_l), j\not=l, j,k\in \{1,\ldots,k\}$
	is finite. 
	
	Let $x \in DIP_1(L).$ 
	In cases (b)(c) it follows that
	if $x=c_j(t)$ for some $j=1\ldots,k$ then $t \in (0,1).$
	In case (b) a geodesic chord
	$c_j:[0,1]\longrightarrow M$ satisfies
	$c_j(t)\not\in \partial M$ for all $t\in (0,1),$
	hence $DIP_1(L)\cap \partial M=\emptyset.$
	
	In case (c) if $x\in DIP_1(L)$ satisfies
	$x=c_j(0)=c_j(t)$ for some $t \in (0,1]$ then 
	$c_j$ lies on a closed geodesic which is not
	possible by assumption. If $x=c_j(0)=c_k(t)$ for
	$j\not=k$ then $t\not\in\{0,1\}$ since otherwise
	the geodesics $c_j,c_k$ lie on the same geodesic
	line $\gamma: \R\longrightarrow M$ since $N$ has codimension one.
	If $x=c_j(0)=c_i(t)$ for $i,j=1,\ldots,k, 0<t<1$
	then we can use the perturbation result
	Lemma~\ref{lem:perturbation} and obtain a 
	Riemannian metric for which the geodesic segment
	$c_j$ in the neighborhood remains unchanged
	and the geodesic segment $c_k$ 
	in the neighboorhood is perturbed 
	such that $c_j$ and $c_k$ do not intersect in the
	neighborhood of $x.$
	
	Now we consider case (a).
	
	We first assume $p=q.$ For a Riemannian metric
	$g \in \mathcal{G}_{\{p\}\times \{p\}}(M)[L]$
	for some $L>0$
	there are only finitely many geometrically distinct
	and non-degenerate 
	geodesic loops $\gamma_1,\ldots,\gamma_N$ of length
	$\le L.$ Assume that - after a possible renumbering -
	the geodesic loops $\gamma_1,\ldots, \gamma_m, m \le N$ are
	the prime geodesic loops of length $\le L.$
	Here a geodesic loop $\gamma: [0,1]/\{0,1\} \longrightarrow M$
	is prime, if there is no $t_1 \in (0,1)$ with
	$p=c(t_1).$ 
	Let $v_{j,l}=\gamma_j'(l), j=1,\ldots,m, l\in \{0,1\}$
	be the sequence of tangent vectors of the prime geodesic
	loops at $p.$ We use the perturbation result
	Lemma~\ref{lem:perturbation2} which shows that
	there is an open and dense subset
	$\mathcal{G}_p[L]\subset \mathcal{G}_{\{p\}\times\{p\}}(M)[L]$
	for which all geodesic loops at $p$ of length $\le L$ are 
	prime. 
	Note here that if $m<N,$ then the geodesic loops
	$\gamma_{m+1},\ldots,\gamma_N$ are not prime, but their
	prime parts are loops $\gamma_j, 1\le j\le m.$ 
	Since after the perturbation these prime geodesic loops
	cannot be concatenated to produce a non-prime geodesic
	loop the perturbed loops $\gamma_j[g],j\ge m+1$ have
	to be prime, too.
	
	In case $p\not=q$ we consider the space
	$\mathcal{G}_{p,q}[L]=\mathcal{G}_p[L]\cap
	\mathcal{G}_q[L]
	\cap \mathcal{G}_{\{p\}\times \{q\}}(M)[L].$ 
	Hence for $g\in \mathcal{G}_{p,q}[L]$ all geodesic loops at $p$ and $q$
	of length $\le L$ are non-degenerate and prime 
	and there are only finitely many geometrically distinct
	geodesic loops at $p$ and $q.$
	And 
	all geodesic
	segments joining $p$ and $q$ of length $\le L$ are
	non-degenerate and there are only finitely many geometrically
	distinct ones. Hence the set of tangent vectors	
	of geodesic loops at $p$ resp. $q$ of length $\le L$
	as well as the set of tangent vectors of prime geodesic segments
	joining $p$ and $q$ of length $\le L$ at $p$ and $q$ is finite.
	Therefore we can use the perturbation lemma~\ref{lem:perturbation2}
	a finite number of times to obtain a metric for which all
	geometric segments joining $p$ and $q$ are prime.
	Here the observation is that a geodesic segment between $p$ and
	$q$ which is not prime contains either distinct prime geodesic segments
	joining $p$ and $q$ or contains prime geodesic segments and
	geodesic loops at $p$ or at $q.$
	
	Therefore this argument shows that for points $p,q$ (here
	we also allow $p=q$) the set 
	$\mathcal{G}'_{\{p\}\times\{q\}}[L]$ of Riemannian metrics 
	for which all geodesic segments 
	of length $\le L$ 
	joining $p$ and $q$ are non-degenerate
	and prime is an open and dense subset.
	If now $x \in DIP_1(L), x\not\in \{p,q\}$ 
	we can choose $\eta\in (0,\inj/3)$ 
	such that the following holds:
	If a geodesic $c_j$ for some $j=1,\ldots,k$
	enters the geodesic ball
	$B_p(2\eta)$ around $p\in DIP(L)$ with radius $2\eta,$
	i.e. $c_j(s) \in B_p(2\eta),$ then there is 
	$s_1$ with $c_j(s_1)=p$ and 
	$d(c_j(s),c_j(s_1))=|s-s_1|\|c_j'\|\le 2\eta.$
	Hence one can use the perturbation result~\cite[Lem.3]{R2022}
	resp. Lemma~\ref{lem:perturbation}
	and perturb the metric in the ball $D_p(2\eta)$ 
	for all finitely many points in $ DIP_1(L)$ such that
	we obtain a metric $g_1\in \mathcal{G}_{\{p\}\times\{q\}}^*(M)[L]$ in
	any arbitrary small neighborhood of the original metric.
\end{proof}
\section{Proofs of the statements in section~\ref{sec:results}}
Theorem~\ref{thm:generic-loop} is Part (a) of Theorem~\ref{thm:four}.
\begin{proof}[Proof of Corollary~\ref{cor:loops}]
	Since $M$ is simply-connected and compact the following holds:
	There is a number
	$l \in \{2,\ldots,n-1\},$ such that the
	sequence $b_{2lk}=\dim H_{2lk}\left(\Omega_{\{p\}\times\{q\}}(M);\q\right),
	k \ge 1$ satisfies $b_{2lk}\ge 1.$
	It was shown by Gromov~\cite{Gromov}
	and Paternain~\cite[Thm.5.1]{Pa} that there is a constant $C(g)>0$
	such that the critical value $\cri (w_k)$ of a non-trivial homology class
	$w_k\in H_{2lk}\left(\Omega_{\{p\}\times\{q\}}(M);\q\right)$
	satisfies $\cri(w_k)\le  C(g) k$ for all $k\ge 1.$
	By Theorem~\ref{thm:generic-loop} we obtain 
	for a generic
	metric $g$ a sequence
	$c_k$ of simple and non-degenerate geodesic segments joining $p$ and $q$
	of index $2lk$ and length $\le C(g)k.$	
	This follows from the Morse inequalities since the energy-functional
	$E:\Omega_{\{p\}\times\{q\}}(M)\longrightarrow \R$
	is a Morse-function.
\end{proof}
Note that a geodesic loop $c:[0,1]/\{0,1\}\longrightarrow M$
with $p=c(0)$ is prime, if there is no $t\in (0,1)$ with $p=c(t).$
Going through the Proof of Theorem~\ref{thm:four} 
and of Corollary~\ref{cor:loops} we find that
in case of a surface, i.e. for $n=2,$
we obtain the following result. Here we essentially use the
Perturbation Lemma~\ref{lem:perturbation2}:
\begin{proposition}
	\label{pro:surface}
	Let $M$ be a compact surface, i.e. $n=\dim M=2$
	and let $p\in M$ be a point. 
	\begin{itemize}
		\item[(a)]
		For a 
		$C^k$-generic Riemannian metric on $M$ with $k\ge 2$
		all geodesic loops based at $p$ are non-degenerate and prime.
		\item[(b)]
		For a $C^k$-generic Riemannian metric on
		the $2$-sphere
		$S^2$ the number
		$N_g(p;t)$ of geodesic loops based at $p$ 
		satisfies
		$\liminf_{t\to \infty}(N_g(p;t)/t)>0.$
	\end{itemize}
\end{proposition}
Theorem~\ref{thm:generic-ogc} is Part (b) of Theorem~\ref{thm:four}.
\begin{proof}[Proof of Theorem~\ref{thm:geod-chords}]
	As remarked in \cite[p.117]{GGP2018} the absence of 
	orthogonal-tangent geodesic chords is an open condition
	with respect to the $C^1$-topology, hence $\mathcal{G}_1$
	is an open subset. For a rotationally symmetric metric on
	the disc there are no orthogonal-tangent geodesic chords
	since any geodesic segment starting orthogonally from $\partial D^n$
	will end orthogonally at the boundary $\partial D^n.$
	Hence $\mathcal{G}_1$ is non-empty. 
	Let $g_0$ be the Euclidean inner product on the
	$n$-disc $D^n\subset \R^n,$ then let 
	$g \in \mathcal{G}(D^n)\longmapsto D(g)\in R^+$ be the smallest
	positive number such that 
	$D(g)^{-1}g_0(v,v)\le g(v,v)\le D(g)g_0(v,v)$
	for any tangent vector $v \in T_xD^n, x\in D^n.$
	Since the mapping $D: \mathcal{G}(D^n)\longrightarrow \R^+$
	is continuous the set
	$\mathcal{G}_0$ of Riemannian metrics for which 
	all orthogonal geodesic chords of length $\le \pi D(g)	$
	are non-degenerate and simple is an open and dense subset by
	Theorem~\ref{thm:four}(b),
	compare the argument in~\cite[Sec.6]{R2023}.
	Then the intersection $
	\mathcal{G}_2=\mathcal{G}_0\cap \mathcal{G}_1$
	contains all metrics $g$ which do not have an orthogonal-tangent
	geodesic chord and for which all orthogonal geodesic chords
	of length $\le D(g)$ are simple and non-degenerate and do not
	intersect each other. Then we use the result by
	Giamb\`o, Giannoni and Piccione~\cite[Thm.1.1]{GGP2018} which ensures
	the existence of $n$ 
	geometrically distinct
	orthogonal geodesic chords for a metric
	without orthogonal-tangent geodesic chords.	
\end{proof}
If the boundary $\partial D^n$ is strictly convex there are no
orthogonal-tangent geodesic chords, hence Theorem~\ref{thm:geod-chords}
implies Corollary~\ref{cor:bos}.



\end{document}